\documentclass{amsart}
\usepackage{amssymb,latexsym,amsfonts,amsthm,amsmath}
\usepackage{enumerate, epsfig}

\newcommand{\ar}{\mbox{area}}
\newcommand{\T}{\mathcal{T}}
\newcommand{\PP}{\mathcal{P}}

\newcommand{\cS}{\mathcal{S}}

\numberwithin{equation}{section}

\theoremstyle{plain}
\newtheorem{theorem}{Theorem}
\newtheorem{corollary}[theorem]{Corollary}
\newtheorem{lemma}[theorem]{Lemma}
\newtheorem{proposition}[theorem]{Proposition}

\theoremstyle{definition}
\newtheorem{definition}[theorem]{Definition}

\theoremstyle{remark}
\newtheorem{remark}[theorem]{Remark}

\begin{document}
\title{Conformal and cp types of surfaces of class $\cS$}
\author{Byung-Geun Oh}
\address{Department of Mathematics Education, Hanyang University, 17 Haengdang-dong, Seongdong-gu, Seoul 133-791, Korea}
\email{bgoh@hanyang.ac.kr}
\date{\today}

\begin{abstract}
In this paper we describe how to define the circle packing (cp) type
(either cp parabolic or cp hyperbolic) of a Riemann surface of class $\cS$, and
study the relation between this type and the conformal type of the surface.
\end{abstract}

\maketitle

\section{Introduction and the definition of cp type}
Let $X$ be a simply-connected Riemann surface, and suppose $h: X \to \overline{\mathbb{C}}$ is a meromorphic function.
The pair $(X,h)$ is called a \emph{Riemann surface of class $\cS$} if there exist $q$ points
$a_1,a_2,\ldots,a_q \in \overline{\mathbb{C}}$ such that the restriction map
\begin{equation}\label{covering}
 h: X \backslash \{h^{-1}(a_j)  : j=1,2,\ldots,q \} \to \overline{\mathbb{C}} \backslash \{a_1,\ldots,a_q\}
\end{equation}
is a topological covering map. In this case  $(X,h) \in F_q (a_1, \ldots, a_q)$ is a common notation, but the simplified notation
$X \in F_q$ is preferred when there is no confusion about the meromorphic function $h$ and the \emph{base}
points $a_1, \ldots, a_q \in \overline{\mathbb{C}}$.
Function theoretically $X \in F_q$ means that all the critical and asymptotic points of $h$ lie over only the finitely many points,
say $a_1, \ldots, a_q$.

Consider a closed Jordan curve $\Gamma'_0 \subset \overline{\mathbb{C}}$, called the \emph{base curve}, passing through
$a_1, a_2, \ldots, a_q$ in this order. Then we can think of $\Gamma'_0$ as a finite connected planar
graph with vertices $a_1, \ldots, a_q$ and edges $[a_i, a_{i+1}]$, $i=1,\ldots,q$ in mod $q$.
Therefore the dual graph of $\Gamma'_0$ is well-defined. Now we denote the dual graph by $\Gamma_0$, and observe that
the pull-back of $\Gamma_0$ via $h$, $\Gamma := h^{-1} (\Gamma_0)$,
is a connected planar graph that is bipartite and homogeneous of degree $q$. This graph $\Gamma$ is called
the \emph{Speiser graph} or the \emph{line complex} of $X \in F_q$.
Conversely, for a given connected planar graph $\Gamma$ that is bipartite and homogeneous of degree $q$,
it is possible to define a Riemann surface of class $\cS$ whose Speiser graph is $\Gamma$.
Furthermore, it is known that for fixed base points $a_1, \ldots, a_q$ and the base curve $\Gamma'_0$,
there is a one-to-one correspondence between (labeled) Speiser graphs and Riemann surfaces of class $\cS$. For more details,
see for example \cite[Chap. XI]{Ne} or \cite{DGP}.

By the famous uniformization theorem \cite[Chap. X]{Ahl}, every simply-connected Riemann surface $X$ is conformally equivalent to
one, and only one, of the following:
the unit disk $\mathbb{D}$, or the whole complex plane $\mathbb{C}$, or the Riemann sphere $\overline{\mathbb{C}}$.
Accordingly we say that the conformal type of $X$ is hyperbolic,
parabolic, or elliptic, respectively. Thus when we study a Riemann surface of class $\cS$, $(X,h) \in F_q$,
the conformal type of $X$ should have been already determined even before the meromorphic
function $h$ was considered. However, this does not mean that $h$ has nothing to do with the type of $X$.
In many cases it is even possible to determine the type of $X$
from the information about $h$. For example, we know from the Picard's theorem that $X$ must be hyperbolic
if $h$ omits three points in $\overline{\mathbb{C}}$.
This kind of process, or problem, that is, determining the conformal type of $X$ using the information about $h: X \to \overline{\mathbb{C}}$, is called the
\emph{type problem}.

The purpose of this paper is to study the type problem for Riemann surfaces of class $\cS$, and compare the conformal type of $X \in F_q$ with
its \emph{circle packing} type (cp type) which we will define later.

In the type problem the elliptic case is often excluded from the beginning and $X$ is assumed to be \emph{open}. This is because
when the Riemann surface is conformally equivalent to $\overline{\mathbb{C}}$,
then it is compact and consequently distinguished from the other two cases very easily.
Thus we always assume
that $X$ is conformally equivalent to either the whole plane $\mathbb{C}$ or the unit disk $\mathbb{D}$.
Moreover, to define the cp type appropriately, we assume that the meromorphic function
$h$ has \emph{no asymptotic values}. Then one can show that the restriction map in \eqref{covering} is
a covering map of \emph{infinite} order, the corresponding Speiser graph $\Gamma$ is an \emph{infinite} graph,
and $\Gamma'$, the dual of $\Gamma$, is nothing but the pull-back graph of $\Gamma'_0$; i.e.,
$\Gamma' = h^{-1} (\Gamma'_0)$. (When $h$ has an asymptotic value, then the pull-back $h^{-1} (\Gamma'_0)$ is not even a graph.
Note that a vertex in $\Gamma'_0$ can be lifted to a point at infinity.)

Let $V$ be an index set, and recall that an indexed \emph{circle packing}
$\PP = (P_v : v \in V)$ in the plane $\mathbb{C}$ is a collection of closed geometric disks with
disjoint interiors. The \emph{contacts graph}, or \emph{nerve}, of a circle packing $\PP$
is a graph whose vertex set is $V$ and such that an edge $[v,w]$ appears in the graph
if and only if $P_v$ and $P_w$ intersects. An \emph{interstice} of $\PP$ is a
connected component of the complement of $\cup_{v \in V} P_v$, and the \emph{carrier}
is the union of the packed disks and the \emph{finite} interstices.

Now we are ready to describe our problem. For a given Riemann surface $(X,h) \in F_q$ of class $\cS$, let $T_0$ be a triangulation of the Riemann sphere $\overline{\mathbb{C}}$ such that all the base points $a_1, \dots, a_q$ are contained in the vertex set of $T_0$.
We further assume that the pull-back graph $T := h^{-1} (T_0)$ is a disk triangulation.
Then there exists a circle packing $\PP$ in $\mathbb{C}$
whose nerve is combinatorially equivalent to (the 1-skeleton of) $T$
and whose carrier is either the whole complex plane $\mathbb{C}$
or the unit disk $\mathbb{D}$ \cite[Corollary 0.5]{HS1}. The graph $T$ is called
\emph{circle packing parabolic} type (cp parabolic) if the carrier of $\PP$ is $\mathbb{C}$,
and otherwise it is called \emph{cp hyperbolic}.
Our question is, does this cp type depend on the triangulation $T_0$ of $\overline{\mathbb{C}}$?

\begin{proposition}\label{T}
Suppose $(X,h)$ is a Riemann surface of class $\cS$ and $T_0$ and $\mathcal{T}_0$ are triangulations of
$\overline{\mathbb{C}}$ whose vertex sets contain all the base points
and such that their pull-back graphs $T := h^{-1} (T_0)$ and $\mathcal{T} :=  h^{-1} (\mathcal{T}_0)$  are disk triangulations.
Then the cp types of $T$ and $\mathcal{T}$ coincide.
\end{proposition}

The condition that $T$ (or $\mathcal{T}$) is a disk triangulation is a minor one which
can be easily satisfied. In fact, if $T_0$ has no edge of the form
$[a_i, a_j]$ for $i,j \in \{1,\ldots, q\}$, then $T$ must be a disk triangulation.
A bad case happens only when we have multiple edges in $T$ between two vertices lying over some base points $a_j$'s
(Figure~\ref{me}).

\begin{figure}
\begin{center}
 \input{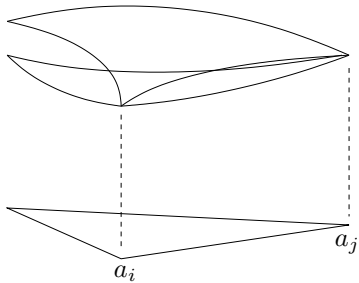}
 \caption{two vertices with a multiple edge}\label{me}
\end{center}
\end{figure}

Proposition~\ref{T} basically deals with the case when $\Gamma'$ is not of bounded valence, since
it is trivially verified when the local degree of $h$ is uniformly
bounded; i.e., when the dual Speiser graph $\Gamma'$ is of bounded valence.
This is because the pull-back triangulation $T$ is
roughly isometric (see  \cite{So}, p.~160, for the definition of rough isometries, which are also called \emph{quasi-isometries}
depending on the literature) to the dual Speiser graph $\Gamma'$,
hence $T$ is recurrent if and only if $\Gamma'$ is recurrent \cite[Theorem~7.18]{So} (we say that a graph is recurrent or transient if
the simple random walk on the graph is recurrent or transient, respectively).
Note that in this case $T$ is also of bounded valence.
Now from \cite[Theorem 1.1]{HS2}, where He and Schramm showed that
a bounded valence disk triangulation graph $T$ is recurrent if
and only if it is cp parabolic, we conclude that $T$ is cp parabolic if and only if $\Gamma'$ is recurrent.
Similarly $\mathcal{T}$ is cp parabolic if and only if $\Gamma'$ is recurrent, hence the cp types of
$T$ and $\mathcal{T}$ must coincide in this case.

In the course of a proof for Proposition~\ref{T} we encountered the following statement, which might be interesting by itself.

\begin{proposition}\label{TT}
Suppose $G$ is a disk triangulation graph and $G'$ is a semi-bounded refinement graph of $G$. Then $G$ is VEL-parabolic if and only if $G'$
is VEL-parabolic.
\end{proposition}

The definitions for VEL-parabolicity and refinement graphs are given in \S\ref{VEL} and \S\ref{Finer}, respectively. After completion of this work
we learned that the same result as Proposition~\ref{TT} was obtained independently by Wood \cite{Wo} using a different method.

Proposition~\ref{T} allows us to define:

\begin{definition}\label{D:cptype}
A Riemann surface $(X,h)$ of class $\cS$ is called cp parabolic if
$T = h^{-1} (T_0)$ is cp parabolic, where $T_0$ is as in Proposition~\ref{T}.
Otherwise it is called cp hyperbolic.
\end{definition}

\begin{remark}
We excluded the elliptic case from our consideration, since in this case everything becomes very trivial. Note that if $X$ is the Riemann sphere,
then $h$ must be a rational map, hence the graph $T = h^{-1} (T_0)$ is a finite graph. Definitely this case is distinguished from the other two cases.
If needed, however, one can define that $(X,h)$ is cp elliptic if $T$ is a finite graph.
\end{remark}

As mentioned earlier, the conformal type of $X$ is uniquely determined by the uniformization theorem.
Then a natural question is,
does the conformal type of $X$ and cp type of $(X,h)$ agree? The answer for this question should be positive under some ``nice" conditions,
since certain finite circle packings approximate
the Riemann map as conjectured by Thurston \cite{Thu} and proved by Rodin and Sullivan \cite{RS}. For example, we have
the affirmative answer if the local degree of $h$ is
uniformly bounded. In this case one can show that $T$ is roughly isometric to the Speiser graph $\Gamma$ as well as the dual Speiser graph $\Gamma'$,
hence $T$ must be cp parabolic if and only if $\Gamma$ is recurrent as discussed preceding Definition~\ref{D:cptype}.
(Note that every Speiser graph is of bounded valence, since it is a homogeneous graph.)
Thus by applying Doyle's criterion \cite{Do} (cf. \cite{Me}) to this case, we see that the Speiser graph $\Gamma$ is
recurrent if and only if $X$ is conformally equivalent to $\mathbb{C}$.
In general, however, the conformal type of $X$ and cp type of $(X,h)$ do not have to match.

\begin{theorem}\label{Example}
There exists a parabolic Riemann surface of class $\cS$ that is cp hyperbolic.
\end{theorem}

If we extend the definition of cp-type to the surfaces with asymptotic values (Remark~\ref{remk}, p.~\pageref{remk}), then
the surface in \cite{GM} serves as a counterexample to the other part of the implication; that is, the surface constructed in Section 2
of \cite{GM}, although it has some asymptotic values, is the example of $X \in F_q$ such that
$X$ is of cp parabolic but conformally equivalent to the disk.

\section{Vertex Extremal Length and Circle Packings}\label{VEL}
Let $G = (V, E)$ be a graph, where $V$ is the vertex set and $E$ is the edge set of $G$. Every edge $e \in E$
is associated with two vertices $v, w \in V$, saying that $e$ is incident to $v$ and $w$, or $e$ connects $v$ and $w$.
In this case we write $e= [v,w]$, and $v$ and $w$ are called the endpoints of $e$. This notation might be confusing, however,
since we allow multiple edges between two vertices. (Most Speiser graphs have multiple edges; for example see Figure~\ref{ES}.)
On the other hand, we always assume that there is no self-loop.

A closed subset of the plane is called a \emph{face} of $G$ if it is the closure of a component of $\mathbb{C} \setminus G$.
Two vertices $v$ and $w$ are  \emph{neighbors} if $[v,w] \in E$, and we denote by $N(v)$ the set of neighbors of $v \in V$.
The \emph{degree} or \emph{valence} of $v \in V$ is defined by the number of edges from $v$ to its neighbors, and
denoted by $\deg (v)$. We say that the graph $G$
is of bounded valence if $\sup \{ \deg (v) : v \in V \} < \infty$.
A \emph{path} $\gamma$ in $G$ is a finite or infinite sequence $[v_0, v_1, \ldots]$ of vertices
such that $[v_i, v_{i+1}] \in E$ for every $i=0,1,\ldots$. We denote by $V(\gamma)$ the set of vertices which the path $\gamma$ visits,
and an infinite path is called \emph{transient} if it visits infinitely many distinct vertices.
A graph is called connected if every two vertices in the graph can be connected by a finite path, and infinite if it contains infinitely
many vertices.
Throughout this paper, we always assume that every graph is connected and infinite, unless otherwise stated. (For given $(X, h) \in F_q$,
the base curve $\Gamma'_0$ and its dual $\Gamma_0$  are finite, though.)

A graph $G$ is called \emph{planar} if there is an embedding $\iota$ from (the 1-complex of) $G$ into $\mathbb{C}$, and its image $\iota(G)$ is
called the embedded graph. In general there are some differences between a planar graph and its embedded graphs, since
there could be two or more topologically different embeddings of the same planar graph,
but we do not distinguish a planar graph from its embedded graph. A planar graph
is called \emph{locally finite} if the cardinality of $V \cap K = \iota(V) \cap K$ is finite for every compact set $K \subset \mathbb{C}$.

\begin{definition}\label{vmetric}
A graph $G = (V,E)$ is called \emph{Vertex Extremal Length parabolic}, or \emph{VEL-parabolic}, if there exists a function
$m : V \to \mathbb{R}^+ \cup \{ 0 \}$ such that
\begin{align*}
 &\mbox{(a)} \,\, \sum_{v \in V} m(v)^2 < \infty, \\
 &\mbox{(b)}  \sum_{v \in V(\gamma)} m(v) = \infty \quad \mbox{for every transient path } \gamma.
\end{align*}
A function $m$ satisfying these properties will be called \emph{parabolic v-metric}.
If there is no such v-metric, $G$ is called VEL-hyperbolic.
\end{definition}

He and Schramm showed in \cite{HS2} that a disk triangulation graph $G$ is
VEL-parabolic if and only if $G$ is cp parabolic. If $G$ is of bounded
valence, the VEL-parabolicity and recurrence of $G$ are equivalent
(\cite{HS2}, Theorem~8.1 and Theorem~2.6).

The following definition is due to Schramm \cite{Sh}.

\begin{definition}
Let $\tau > 0$. A Lebesque measurable set $A \subset \mathbb{C}$
is called \emph{$\tau$-fat} if for every $x \in A$ and $r >0$ such that
$D(x,r) := \{ z : |z-x| < r \}$ does not contain $A$, the following inequality holds:
\[
\mbox{area}(A \cap D(x,r)) \geq \tau \cdot \mbox{area}(D(x,r)).
\]
\end{definition}

Here $\mbox{area}(\cdot)$ denotes the 2-dimensional Lebesque measure.
A typical example of a fat set is an Euclidean disk. For instance, if $D = D(0,1)$, $0 \leq x < 1$, and
$0 < r < 1+x$, then $D(x - r/2, r/2) \subset D \cap D(x, r)$.
Therefore,
\[
 \ar (D \cap D(x,r)) \geq \ar (D(x- r/2, r/2)) = (1/4) \cdot \ar (D(x,r)),
\]
showing that $D$ is 1/4-fat.

\begin{lemma}\label{union}
Suppose $A$ and $B$ are $\tau$-fat sets for some $\tau > 0$ and
$A \cap B \ne \emptyset$. Then $A \cup B$ is $\tau /4$-fat.
\end{lemma}

\begin{proof}
Suppose $x \in A \cup B$. Without loss of generality, we assume that $x \in A$.
If $D(x, r/2)$ does not contain $A$, then
\begin{align*}
 \mbox{area}((A \cup B) \cap & D(x,r)) \geq \ar(A \cap D(x,r))
    \geq \ar(A \cap D(x,r/2)) \\
 & \geq \tau \cdot \ar (D(x, r/2)) = (\tau /4) \cdot \ar (D(x,r))
\end{align*}
as desired.

If $D(x, r/2)$ contains $A$ but $D(x, r)$ does not contain $A \cup B$,
there exists $y \in A \cap B \subset B$ such that
$|x-y| < r/2$. Since $D(y, r/2) \subset D(x,r)$,
$A \subset D(x, r/2)$, and $A \cup B \nsubseteq D(x,r)$, one can easily see that
$D(y, r/2)$ does not contain $B$. Therefore,
\begin{align*}
 \mbox{area}( (A \cup B) \cap & D(x,r)) \geq \mbox{area}( B \cap D(x,r)) \geq
 \ar(B \cap D(y,r/2)) \\
 & \geq \tau \cdot \ar (D(y, r/2)) = (\tau /4) \cdot \ar (D(x,r)),
\end{align*}
which completes the proof.
\end{proof}

We conclude this section with the following lemma.

\begin{lemma}[Lemma 3.4 of \cite{HS2}]\label{HS}
Suppose $G = (V,E)$ is a locally finite planar graph with infinitely many vertices, and $\PP = (P_v : v \in V)$ is
a collection of $\tau$-fat sets satisfying the following properties:
\begin{enumerate}[(1)]
 \item for every $v \in V$, $P_v$ is a compact connected set in $\mathbb{C}$;
 \item $\PP$ is locally finite in $\mathbb{C}$; that is, for every compact set $K \subset \mathbb{C}$,
       there are only finitely many $v \in V$ such that $P_v \cap K \ne \emptyset$;
 \item every $x \in \mathbb{C}$ is contained in $P_v$ for at most $M < \infty$
       vertices $v \in V$, where $M$ does not depend on $x$;
 \item if $[v,w] \in E$, then $P_v \cap P_w \ne \emptyset$.
\end{enumerate}
Then $G$ is a VEL-parabolic graph.
\end{lemma}

In fact, He and Schramm showed the above lemma
when $\PP$ is a $\tau$-fat \emph{packing}, but one can
easily check that their proof also works in this case.

\section{Refinement Graphs and VEL-Parabolicity}\label{Finer}
 Let $G$ and $G'$ be two locally finite planar graphs.
 We say that $G'$  is a \emph{refinement} graph of $G$ if every face and edge of $G$
 are unions of a finite number of faces and edges, respectively, of $G'$.
 A refinement graph $G'$ of $G$ is called \emph{bounded} (or \emph{semi-bounded}) if there exists an absolute constant $M>0$ such that
 every face (or every edge, respectively) of $G$ contains at most $M$ vertices of $G'$.

 \begin{lemma}\label{L:finer}
 Suppose $G' = (V', E')$ is a semi-bounded refinement graph of $G = (V, E)$.
 If $G'$ is VEL-parabolic, so is $G$.
 (Equivalently, if $G$ is VEL-hyperbolic, so is $G'$.)
 \end{lemma}

 \begin{proof}
  For every given $v \in V$, we define the starlike set centered at $v$ by
 \[
 E_v := \left( \bigcup_{w \in N(v)} [v,w] \right) \backslash N(v).
 \]
 Here $N(v)$ is the set of neighbors in $G$, not in $G'$, and note that $E_v$ is a subset contained in (the 1-skeleton of) $G$.
 For future reference, we also let $\overline{E}_v = E_v \cup N(v)$.

 Suppose that $m'$ is a parabolic v-metric that is obtained from the VEL-parabolicity of $G'$.
 Since $G'$ is a semi-bounded refinement graph of $G$, there are at most $M>0$ vertices of $G'$ that
 are contained in one edge of $G$.
 Let
 \[
 m(v) := 2 M \cdot \max \{ m'(w) : w \in E_v \cap V' \} \qquad \mbox{for all } v \in V,
 \]
 and we claim that $m$ is a parabolic v-metric of $G$.

 First we show that $m$ is square-summable. But this is easy because every vertex $w  \in V'$ is
 contained in $E_v$ for at most two $v \in V$, thus we have
 \begin{align*}
  \sum_{v \in V} & m(v)^2
  = \sum_{v \in V} 4 M^2 \cdot \max \{ m'(w)^2 : w \in E_v \cap V' \} \\
 & \leq 4 M^2 \sum_{v \in V} \left( \sum_{w \in E_v \cap V'} m'(w)^2 \right)
   \leq 8 M^2 \sum_{w \in V'} m'(w)^2 < \infty,
 \end{align*}
 as desired.

 Now suppose $e = [u,v]$ is an edge of $G$. Then because $e \subset E_u \cup E_{v}$,
 \begin{align*}
  & ~~~ \sum_{w \in (e \cap V')}  m'(w) \leq M \cdot \max \{ m'(w) : w \in (E_u \cup E_{v}) \cap V' \}\\
  & \leq  M \cdot \max \{ m'(w) : w \in E_u \cap V'\} + M \cdot \max \{ m'(w) : w \in E_{v} \cap V' \} \\
  & = \frac{1}{2} \{m(u) + m(v)\}.
 \end{align*}
 Therefore, because any transient path $\gamma = [v_0, v_1, v_2, \ldots]$ in $G$
 also can be realized as a transient path
 $\gamma' = [v_0, v_0^1, \ldots v_0^{k_0}, v_1, \ldots, v_2, \ldots]$ in $G'$, we have
 \begin{align*}
 \infty = & \sum_{w \in V(\gamma')} m'(w) \leq
 \sum_{i=0}^\infty \left( \sum_{w \in [v_i, v_{i+1}] \cap V'} m'(w) \right) \\
 &\leq \sum_{i=0}^\infty \frac{1}{2} \{m(v_i) + m(v_{i+1})\}
 \leq \sum_{v \in V(\gamma)} m(v),
 \end{align*}
 which completes the proof.
 \end{proof}

Suppose $G=(V,E)$ is a planar graph and let $K$ be a positive integer.
We say that $G$ satisfies the property $p(K)$
if $\min \{ \deg (v), \deg(w) \} \leq K$ for all
$[v,w] \in E$.

\begin{lemma}\label{L:cfiner}
Suppose $G = (V, E)$ is a disk triangulation graph satisfying the property
$p(K)$ for some $K > 0$ and let $G' = (V', E')$ be a semi-bounded refinement graph
of $G$. If $G$ is VEL-parabolic, then $G'$ is also VEL-parabolic.
(Equivalently, VEL-hyperbolicity of $G'$ implies VEL-hyperbolicity of $G$.)
\end{lemma}

\begin{proof}
For $x \in \mathbb{C}$ contained in the 1-skeleton of $G$, we will use the expression  $x \in G$.
Also for $w \in (G \cap V')$, we define
$V_w = \{v \in V : w \in \overline{E}_v\}$. Note that if $w \in (G\cap V') \setminus V$, then $V_w$ consists of two vertices of $V$,
the endpoints of the edge containing $w$, while $V_w = N(w) \cup \{ w \}$ if $w \in V$.
Also let $Z := \{ v \in V : \deg (v) > K \}$. Here $\deg (v)$ is the number of edges incident to $v$ in the graph $G$, not in $G'$,
but we will consider $Z$ a subset of $V'$ as well as $V$, since $V \subset V'$.

Let $m$ be the v-metric defined on $V$, which is assumed to exist by the VEL-parabolicity of $G$,
and we define for all $w \in V'$
\[
   m'(w) :=
 \begin{cases}
 m(w),  &\mbox{if } w \in Z;\\
 3 \cdot \max \{m(v) : v \in V_w \setminus Z \}, &\mbox{if } w \in G \setminus Z;\\
 0, &\mbox{otherwise}.
 \end{cases}
\]
Note that  $V_w \setminus Z$ is always nonempty since $G$ satisfies the property $p(K)$. Consequently,
$m'$ is well-defined.

There are at most $M>0$ vertices of $G'$ contained in one edge of $G$,
since $G'$ is a semi-bounded refinement graph of $G$.
Thus if $v \in (V \backslash Z)$, $\overline{E}_v$ contains at most $KM$ vertices of $G'$.
In other words, every $v \in (V \backslash Z)$ is contained in $V_w$ for at most $KM$
different vertices $w \in G \cap V'$. Therefore,
\begin{align*}
 \sum_{w \in V'} m'(w)^2 &= \sum_{w \in Z} m'(w)^2 +
                                \sum_{w \in (G \cap V') \backslash Z} m'(w)^2  + \sum_{w \in V' \setminus G} m'(w)^2\\
 &\leq \sum_{v \in Z} m(v)^2 + 9 \sum_{w \in (G \cap V') \backslash Z}
   \left( \sum_{v \in (V_w \backslash Z)} m(v)^2 \right) + 0\\
 &\leq \sum_{v \in Z} m(v)^2 + 9  KM \sum_{v \in (V \backslash Z)} m(v)^2 \\
 &\leq 9KM \sum_{v \in V} m(v)^2 < \infty,
\end{align*}
which shows that $m'$ is square summable.

Let $F$ be the face set of $G$, and for each $v \in Z$ let $F_v$ be  the union of the faces of $G$ containing $v$ on their boundaries.
We define
\[
 \mathfrak{F}_Z := \{ F_v : v \in Z \} \cup \{ f : f \in F \mbox{ and } f \nsubseteq F_v \mbox{ for any } v \in Z \}.
\]
In other words, an element of $\mathfrak{F}_Z$ is a face of $G$ if none of the vertices on its boundary belongs to $Z$,
and other elements of $\mathfrak{F}_Z$ are $F_v$'s for $v \in Z$. Definitely the interiors of the elements in $\mathfrak{F}_Z$ are mutually disjoint
 and we have $\bigcup_{f \in \mathfrak{F}_Z} f = \mathbb{C}$.
Now suppose $\gamma' = [w_0, w_1, \ldots]$ is a transient path in $G'$.
Let $\{ w_0 = w_{i_0}, w_{i_1}, w_{i_{2}}, \ldots \}$ be a subset of $G \cap V(\gamma') \subset G \cap V'$ whose elements are indexed so that
the finite sub-path $\gamma'_k = [w_{i_k}, w_{i_k +1}, w_{i_k +2}, \ldots, w_{i_{k +1}-1}, w_{i_{k+1}}]$
is contained in $f_k$ for some $f_k \in \mathfrak{F}_Z$. Also note that $w_{i_k}, w_{i_{k+1}} \in \partial f_k$,
where $\partial f_k$ denotes the boundary of $f_k$. 

For each $k$, we will find a compact set $\Lambda_k$  which is contained in the 1-skeleton of $G$ and
whose $m$-length is comparable to the $m'$-length of $\gamma_k'$.
Let $\Lambda_k = \partial f_k$  if $f_k$ is a triangle;
i.e., if $\partial f_k \cap Z = \emptyset$. If $f_k = F_v$ for some
$v \in Z$, we first consider the case that the subarc $\gamma'_k$ passes through
the vertex $v$. In this case, we define $\Lambda_k$ as the
union of all edges of $G$ with one end at $v$ and the other end at
a vertex in $\partial F_v \cap (V_{w_{i_k}}  \cup V_{w_{i_{k+1}}})$ (Figure~\ref{Lambda}). Note that there are at most six such
edges, and $\Lambda_k$ is a union of exactly six edges only when
both $w_{i_k}$ and $w_{i_{k+1}}$ are in $V$. Finally if $f_k = F_v$ for
some $v \in Z$ but the subarc $\gamma'_k$ does not pass through $v$,
there exist $f^1, f^2, \ldots f^l \in F$ which are contained in $F_v = f_k$
such that $f^j \cap \gamma'_k \ne \emptyset$, $j = 1,2, \ldots, l$.
In this case we define $\Lambda_k = \partial F_v \cap (\partial f^1 \cup \partial
f^2 \cup \cdots \cup \partial f^l)$.

\begin{figure}
\begin{center}
 \input{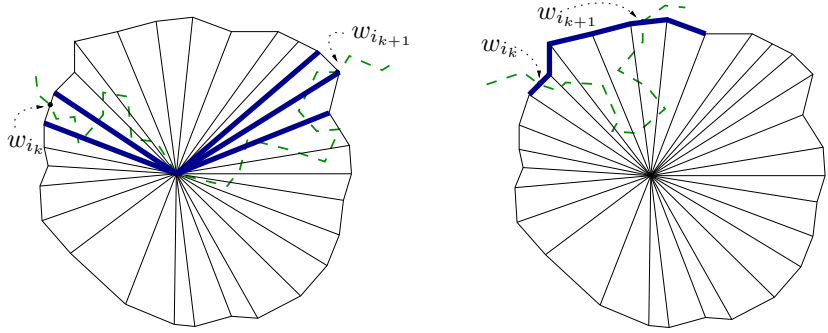}
 \caption{the cases when $\gamma'_k$ passes through $v \in Z$ (left) and passes near but not through $v \in Z$ (right)}\label{Lambda}
\end{center}
\end{figure}

From the definition of $m'$ and $\Lambda_k$, it is easy to see that
\begin{equation}\label{E:comparison}
 \sum_{v \in (\Lambda_k \cap V)} m(v) \leq \sum_{w \in (\gamma'_k \cap V')} m'(w) \leq \sum_{j= {i_k}}^{i_{k+1}} m'(w_j)
\end{equation}
for all $k = 0, 1, 2, \ldots$. Furthermore,
$\Lambda := \bigcup_{k=1}^\infty \Lambda_k$ is a connected unbounded set,
hence there exists a transient path $\gamma \subset \Lambda \cap G$.
Therefore by \eqref{E:comparison},
\begin{align*}
 \infty = &  \sum_{v \in V(\gamma)} m(v) 
        \leq \sum_{k=0}^\infty \left( \sum_{v \in (\Lambda_k \cap V)} m(v) \right) \\
        & \leq \sum_{k=0}^\infty \left( \sum_{w \in (\gamma'_k \cap V')} m'(w) \right)
        \leq 2 \sum_{w \in V(\gamma')} m'(w),
\end{align*}
which completes the proof.
\end{proof}

\section{Proof of Propositions~\ref{T} and \ref{TT}}
Suppose $G = (V,E)$ is a disk triangulation graph. We divide each face of $G$ into
four triangles by connecting the midpoints of the edges, and get a new disk triangulation graph $G^f$. 
Formally, the vertex set of $G^f$ is $V \sqcup E$, the disjoint union of $V$ and $E$,
and an edge $[v,w]$ appears in $G^f$ if and only if (i) $v,w$ are two different edges
of $G$ belonging to the boundary of the same face of $G$; or (ii) $v \in V$, $w \in E$, and $v$ is an end point
of $w$; or (iii) $v \in E$, $w \in V$, and $w$ is an end point of $v$.
Trivially $G^f$ is a disk
triangulation and refinement graph of $G$.
Also note that $G^f$ satisfies the property $p(K)$ with $K=6$.

\begin{lemma}\label{sfiner}
Suppose $G = (V,E)$ is a disk triangulation graph
and $G^f = (V^f, E^f)$ is as above.
Then $G$ is VEL-parabolic if and only if $G^f$ is VEL-parabolic.
\end{lemma}

\begin{proof}
Because $G^f$ is a (semi-)bounded refinement graph of $G$, VEL-parabolicity of
$G^f$ implies VEL-parabolicity of $G$ by Lemma~\ref{L:finer}.

Conversely, suppose $G$ is VEL-parabolic. Then by Corollary~0.5 of \cite{HS1}
and Theorem~1.2 of \cite{HS2},
there exists a circle packing $\PP = (P_v : v \in V)$ whose
nerve is combinatorially equivalent to $G$ and whose carrier is $\mathbb{C}$.
Therefore $G$ can be embedded in $\mathbb{C}$ so that each $v \in V$ is
the center of $P_v$ and each edge $[v,w] \in E$ is a straight line segment
connecting the centers of $P_v$ and $P_w$.

In this embedding, each face of $G$ is a Euclidean triangle, and
if $f$ is a face of $G$ with vertices $u,v,w$,
then the inscribed circle of $f$ passes through the
points $P_u \cap P_v$, $P_v \cap P_w$, and $P_w \cap P_u$ (Figure~\ref{incircle}). Therefore,
if $e=[v,w]$ is an edge of $G$ and $f_1$ and $f_2$ are the faces of $G$
sharing $e$, the union of closed inscribed disks of $f_1$ and $f_2$
is a connected compact set in $\mathbb{C}$, because these two disks
meet at $P_v \cap P_w$. We denote by $P_e$ the union of these two disks.
\begin{figure}[bht]
\begin{center}
\input{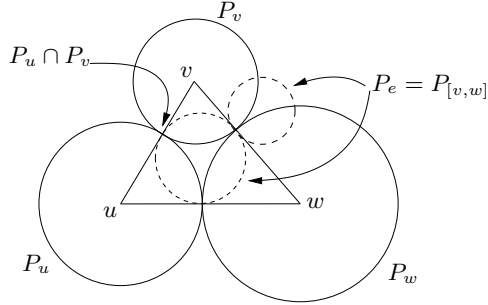}
\caption{inscribed circles and packed disks}\label{incircle}
\end{center}
\end{figure}

Now let $\PP' = (P_w : w \in V^f = V \sqcup E)$; i.e., for $w \in V$ we assign the disk $P_w$ of the circle packing $\PP$,
and for $w \in V^f \setminus V = E$ we assign the union of the inscribed disks tangent to $w$. Then because each disk is (1/4)-fat,
every $P_w$ is (1/16)-fat by Lemma~\ref{union}. Moreover, for every
$x \in \mathbb{C}$ one can easily check that there are at most seven vertices $w \in V^f$
such that $x \in P_w$. (Seven overlapping can actually occur when
$\{ x \} = P_u \cap P_v$ for some $[u,v] \in E$.) It is also trivial that $\PP'$ is locally finite and
that if $[w,w'] \in E^f$ then $P_w \cap P_{w'} \ne \emptyset$.
Thus by Lemma~\ref{HS}, we conclude that $G^f$ is VEL-parabolic.
\end{proof}

\begin{proof}[Proof of Propositions~\ref{T} and \ref{TT}]
We first prove Proposition~\ref{TT}. Suppose $G$ is a disk triangulation graph and $G'$ is a semi-bounded refinement graph of $G$.
If $G'$ is VEL-parabolic,   so is $G$ by Lemma~\ref{L:finer}. Conversely, suppose $G$ is VEL-parabolic.
Let $G^f$ be the graph as in Lemma~\ref{sfiner}, and $G''$ a common semi-bounded refinement graph of both $G^f$ and $G'$.
Then Lemma~\ref{sfiner} implies that  $G^f$ is VEL-parabolic, hence we know from Lemma~\ref{L:cfiner} that
$G''$ is also VEL-parabolic, since $G^f$ satisfies the property $p(6)$.
Now the VEL-parabolicity of $G'$ follows from the VEL-parabolicity of $G''$ and Lemma~\ref{L:finer}.
This completes the proof of Proposition~\ref{TT}.

Proposition~\ref{T} is actually an easy corollary of Proposition~\ref{TT}.
Suppose $T_0$ and $\T_0$ are finite triangulations of $\overline{\mathbb{C}}$
containing all the base points of $(X,h)$ in their vertex sets and
such that their pull-back graphs $T = h^{-1}(T_0)$ and $\T = h^{-1}(\T_0)$
are disk triangulation graphs. Let $\Lambda_0$ be a \emph{finite} common refinement graph of $T_0$ and $\T_0$, and note that
$\Lambda_0$ must be a (semi-)bounded refinement graph of both $T_0$ and $\T_0$ since it is finite.
Then the pull-back graph $\Lambda = h^{-1} (\Lambda_0)$ is a semi-bounded refinement graph of $T$ and $\T$.
Then $T$ is VEL-parabolic if and only if $\Lambda$ is VEL-parabolic by Proposition~\ref{TT},
and similarly $\T$ is VEL-parabolic if and only if $\Lambda$ is VEL-parabolic. Since VEL-parabolicity is
equivalent to cp parabolicity for disk triangulation graphs, this proves Proposition~\ref{T}.
\end{proof}

Observing the above proof, we can get the following additional result. For given $(X, h) \in F_q$, let $\Gamma_0'$ be the
base curve as defined in the introduction, and $\Gamma_0$ its dual. Then the graph
$\Gamma_0$ has exactly two vertices, say $\circ$ and $\times$,
and the vertices of $\Gamma_0'$ is nothing but the base points
$a_1, \ldots, a_q$ of $(X,h)$. For each $j=1,\ldots,q$,
we draw a Jordan arc $\ell_j$ from $\circ$ to $\times$ so that it intersects $\Gamma_0 \cup \Gamma_0'$ only
at $a_j$ and the endpoints $\circ, \times$. Now we define $\Lambda_0$ by
\[
 \Lambda_0 := \Gamma_0 \cup \Gamma_0' \cup \ell_1 \cdots \cup \ell_q.
\]
Then $\Lambda_0$ is a triangulation of $\overline{\mathbb{C}}$ which contains
all the base points in its vertex set, and the pull-back graph
$\Lambda := h^{-1} (\Lambda_0)$ is a disk triangulation graph.

Note that the graph $\Lambda$ can be obtained from the Speiser graph $\Gamma = h^{-1} (\Gamma_0)$ or its dual graph $\Gamma' = h^{-1} (\Gamma'_0)$
by dividing each face with $k$ edges into $2k$ triangles; i.e., if $f$ is a face of $\Gamma$ or $\Gamma'$,
then we pick point in the interior of $f$, and connect it
to the vertices on $\partial f$ and the midpoints of the edges surrounding $f$, so that
$f$ is divided into $2k$ triangles. Therefore one can see that $\Lambda$ is
a semi-bounded refinement graph of $\Gamma$ and $\Gamma'$. (In fact, $\Lambda$ is a bounded refinement graph of $\Gamma'$ since every face of $\Gamma'$ is
a $q$-gon. Also note that $\Lambda$ satisfies the property $p(K)$ with $K = \max \{ 4, 2q \} = 2q$.)
Now by Lemma~\ref{L:finer}, VEL-hyperbolicity of $\Gamma$ or $\Gamma'$
implies that of $\Lambda$, and we have:

\begin{corollary}\label{hyperbolic}
Suppose $(X,h)$ is a Riemann surface of class $\cS$. If either the Speiser graph corresponding
to $(X,h)$ or its dual is VEL-hyperbolic, $(X,h)$ is of cp hyperbolic type.
\end{corollary}

\begin{remark}\label{remk}
It would be more natural to define cp type for every surfaces of class $\cS$, even in the presence of
asymptotic values. Thus suppose $(X,h) \in F_q$ and let $\Gamma$ be its corresponding Speiser graph. The graph $\Gamma$
would have some infinite faces when $h$ has asymptotic values, but we can still divide each \emph{finite}
face of $\Gamma$ into triangles, and obtain a graph similar to $\Lambda$ described before Corollary~\ref{hyperbolic}.

The obtained graph, which we still denote by $\Lambda$, is no longer a disk triangulation graph, but
it is possible to check if $\Lambda$ is VEL-parabolic or hyperbolic.
Thus we can naturally extend the definition of cp types, even in the presence of asymptotic values,
according to the cp type of $\Lambda$; i.e., we can define that $(X,h) \in F_q$ is cp parabolic if and only if $\Lambda$ is
VEL-parabolic.
\end{remark}

\section{Proof of Theorem~\ref{Example}}\label{S:Ex}
Suppose a Speiser graph $\Gamma$ is given. If $\Gamma$ is VEL-hyperbolic, then Corollary~\ref{hyperbolic} implies that
the corresponding surface $X \in F_q$ is cp hyperbolic. On the other hand, since every Speiser graph is of bounded valence,
VEL-hyperbolicity of $\Gamma$ is equivalent to its transiency, which then implies hyperbolicity of $X$
(\cite{Do} or \cite{Me}; see the Doyle's criterion below).  Therefore the conformal and cp types agree in this case.

If $\Gamma$ is VEL-parabolic (hence recurrent), however,
the cp type of $X$ could be either one. Note that
the disk triangulation graph $\Lambda$ described before Corollary~\ref{hyperbolic} is a semi-bounded refinement graph of $\Gamma$,
but the proof for Lemma~\ref{L:cfiner} does not work in this case since $\Gamma$ is not a triangulation.
Definitely neither does the VEL-parabolicity of $\Gamma$ guarantee the parabolicity of $X$ in the conformal type,
so one has to investigate recurrent Speiser graphs for a Riemann surface of class $\cS$ with different cp and conformal types.
Our strategy  is to construct a Speiser graph whose growth rate is very slow
so that the corresponding surface is conformally parabolic, while its dual graph is VEL-hyperbolic. Then
the surface will be cp hyperbolic by Corollary~\ref{hyperbolic}.

First, we start with a Speiser graph $\Psi = (V_\Psi, E_\Psi)$ such that both the graph $\Psi$
and its dual graph $\Psi'$ are VEL-hyperbolic. More specifically, for $\Psi$ we pick  the regular graph of degree 3
such that every face of it is an octagon, so that the dual graph $\Psi'$ is a regular graph of degree 8 which is
a disk triangulation. Fix a vertex $v_0$ of $\Psi$, and let $S_\Psi (n)$ be
the combinatorial sphere of radius $n$ and centered at $v_0$; that is, $S_\Psi (n)$ is the
set of vertices in $V_\Psi$ whose combinatorial distance from $v_0$ is equal to $n$.
Let $B_\Psi (n) = \bigcup_{k=0}^n S_\Psi (k)$ be the combinatorial ball of radius $n$ and centered at $v_0$,
and let $E_\Psi (n)$ be the set of edges in $E_\Psi$
with one end on $S_\Psi (n)$ and the other end on $S_\Psi (n+1)$.

We next choose a sequence of natural numbers (odd numbers) $\{l_n\}_0^\infty$ which will be determined later. Then
for each edge in $E_\Psi (0)$, we replace it by an unbranched tree of length $l_0$
so that the graph remains to be a Speiser graph.
Note that every Speiser graph is bipartite, so the vertex set $V_\Psi$ is divided into two classes,
say vertices tagged with $\circ$ and those with $\times$,
and every edge must connect two vertices with different tags.
Moreover, $\Psi$ is the regular graph of degree 3 with octagonal faces, thus in order
to make the modified graph remain to be a Speiser graph, the number $l_0$ must be odd and
every second edge in this unbranched tree of length $l_0$ must be a double edge (Figure~\ref{br}).
\begin{figure}[h!]
\begin{center}
\includegraphics{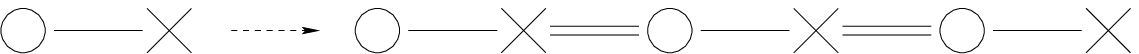}
\caption{replacing an edge with an unbranched tree of length 5}\label{br}
\end{center}
\end{figure}
Similarly we replace every edge in $E_\Psi (1)$ by an unbranched tree of length $l_1$ so
that the graph remains to be a Speiser graph, and repeat this process for all edges in $E_\Psi (n)$, $n=2,3,\ldots$.

The resulting graph $\Gamma = (V_\Gamma, E_\Gamma)$ we have just obtained from this process is a Speiser graph whose
dual graph $\Gamma'$ is VEL-hyperbolic. This is because $\Gamma'$ contains $\Psi'$ as a subgraph,
and we chose $\Psi'$ to be VEL-hyperbolic.
We conclude that the surface $X \in F_q$ associated
with $\Gamma$ is cp hyperbolic by Corollary~\ref{hyperbolic}.

It remains to show that the surface $X \in F_q$ associated with $\Gamma$
is conformally parabolic, for which we use the criterion of Dolye (\cite{Do}; cf. \cite{Me}): `the Riemann
surface of class $\cS$ associated with a given Speiser graph is parabolic
if and only if the corresponding \emph{extended Speiser graph} is recurrent'.
Note that the extended Speiser graph of a given Speiser graph is a planar graph with \emph{infinitely many ends}, obtained
by adding infinitely many square grids to all the faces of the Speiser graph (Figure~\ref{ES}).

\begin{figure}[hbt]
\begin{center}
\includegraphics{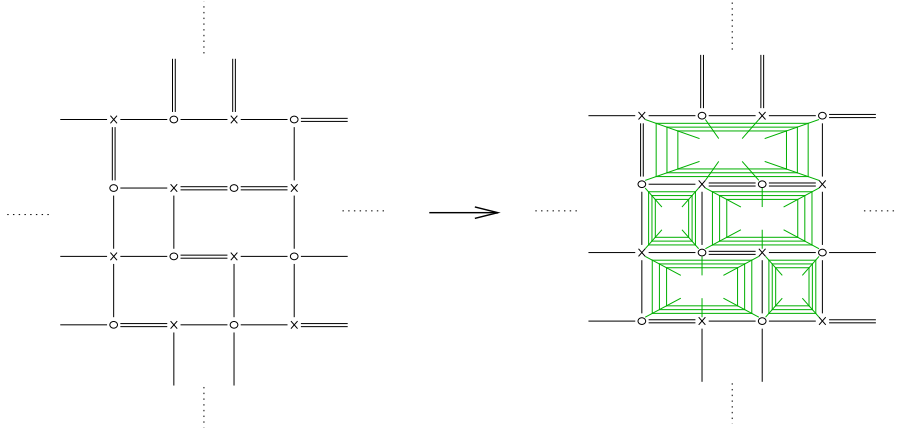}
\caption{a Speiser graph of degree 4 and its extended graph, where the extended part in each bigons(2-gons) are omitted}\label{ES}
\end{center}
\end{figure}

Let $\Upsilon = (V_\Upsilon, E_\Upsilon)$ be the extended Speiser graph obtained from $\Gamma$, and note that the vertex
$v_0$ of $\Psi$, the center of the spheres and balls in the construction of $\Gamma$,
may be regarded as a vertex in $\Gamma$, or even as a vertex in $\Upsilon$. Now we
let $S_\Gamma (n)$, $S_{\Upsilon} (n)$, $B_\Gamma (n)$, $B_{\Upsilon} (n)$
be combinatorial spheres and balls in $\Gamma$ and $\Upsilon$, respectively, with centered at $v_0$ and radius $n$.

Depending on the growth rate of $|B_\Psi (n)|$, the number of vertices in $B_\Psi (n)$, we
can choose the sequence $l_n$ increasing so fast that $|B_\Gamma (k)| \leq k \log k$ for sufficiently
large $k$. More specifically, we have $| S_\Psi(n) | \leq 3^{n}$ since $\Psi$ is a regular graph of degree 3.
This means that $|S_\Gamma (k) | \leq 3^{n}$ for $k \leq \sum_{j=0}^{n} l_j$.
Thus if we define $l_n := \exp(3^{n+1})$,  then for  $1+ \sum_{j=0}^{n-1} l_j \leq k \leq  \sum_{j=0}^{n} l_j$ we have
\[
|B_\Gamma (k)| \leq k 3^n  = k \log l_{n-1} \leq k \log k.
\]

On the other hand,  $\Upsilon$ is a graph whose vertices have degree at most $6$ such that
for each $m \in \mathbb{N}$ and $w \in V_\Gamma$
there are exactly $3$ vertices in $V_\Upsilon \setminus V_\Gamma$ which are directly \emph{$m$ steps above} $w$.
In other words, if $w$ has distance $m-k$ from $v_0$, then only 3 vertices in $V_\Upsilon \setminus V_\Gamma$
are directly above $w$ and have distance $m$ from $w$. Therefore we must have
\[
|S_{\Upsilon} (k) | \leq |B_\Gamma (k)| + |S_{\Upsilon} (k) \setminus B_\Gamma (k) |
\leq k \log k + 3 k \log k = 4 k \log k
\]
for sufficiently large $k$. Therefore there exists a constant $C$ such that
$|B_\Upsilon (k)| \leq C k^2 \log k$ for sufficiently large $k$.
Since $\Upsilon$ is of bounded valence, the Nash-Williams criterion (cf. \cite[p. 56]{So})
shows that $\Upsilon$ is recurrent. We conclude that the surface $X \in F_q$ associated with $\Gamma$ is parabolic
by the Doyle's criterion mentioned above, and this completes the proof of Theorem~\ref{Example}.


\end{document}